\documentclass[preprint,12pt,letterpaper]{elsarticle}
\usepackage[T1]{fontenc}
\usepackage{lmodern}
\usepackage{microtype}
\usepackage[margin=1.0in]{geometry}
\usepackage{amsmath,amssymb,amsthm,mathtools}
\usepackage{xcolor}
\usepackage{booktabs,needspace}
\usepackage[hidelinks]{hyperref}
\numberwithin{equation}{section}
\newtheorem{theorem}{Theorem}[section]
\newtheorem{proposition}[theorem]{Proposition}
\newtheorem{lemma}[theorem]{Lemma}
\theoremstyle{remark}
\newtheorem{remark}[theorem]{Remark}
\newenvironment{bluechange}{\begingroup\color{black}}{\endgroup}
\newcommand{\R}{\mathbb R}
\newcommand{\dd}{\,\mathrm d}
\newcommand{\one}{\mathbf 1}
\newcommand{\zero}{\mathbf 0}
\newcommand{\bfF}{\mathbf F}
\newcommand{\bfD}{\mathbf D}
\newcommand{\bfS}{\mathcal S}
\newcommand{\bfR}{\mathcal R}
\newcommand{\bfQ}{\mathcal Q}
\newcommand{\bfZ}{\mathbf Z}
\newcommand{\bfz}{\mathbf z}
\newcommand{\bfs}{\mathbf s}
\newcommand{\bfr}{\mathbf r}
\newcommand{\bfw}{\mathbf w}
\newcommand{\bfy}{\mathbf y}
\newcommand{\bfW}{\mathbf W}
\newcommand{\norm}[1]{\lVert #1\rVert}
\newcommand{\ceil}[1]{\left\lceil #1\right\rceil}
\newcommand{\LDz}{\boldsymbol z}
\newcommand{\LDZ}{\boldsymbol Z}
\newcommand{\LDD}{\mathsf D}
\DeclareMathOperator{\diag}{diag}
\begin{document}
\begin{frontmatter}
\title{\fontsize{15}{18}\selectfont Sharp Thresholds for Linear Determinacy and Zero Wave Speed in Lotka–Volterra Competition–Diffusion Systems}
\author[a,b]{Shizhao Ma}
\author[a]{Dongyuan Xiao}
\author[c]{Maolin Zhou}
\address[a]{School of Mathematical Sciences, Shanghai Jiao Tong University, Shanghai 200240, P.R. China}
\address[b]{Institute of Natural Sciences, Shanghai Jiao Tong University, Shanghai 200240, P.R. China}
\address[c]{Chern Institute of Mathematics, Nankai University, Tianjin 300071, P.R. China}
\begin{abstract}
We study two sharp competition thresholds in the Lotka–Volterra competition–diffusion system: the zero-speed threshold for bistable traveling fronts and the linear-determinacy threshold for monostable traveling fronts. Building on our previous results on propagation direction and the classification of traveling waves, we derive explicit Gaussian integral representations of both critical parameters, without involving unknown traveling-wave profiles or speeds. These formulas provide exact characterizations of the optimal competition parameters separating the corresponding propagation regimes. The derivation of these representations was assisted by ChatGPT (OpenAI) Pro, while the underlying theoretical foundations were established in our earlier works.
\end{abstract}
\end{frontmatter}

\section{Introduction and main results}
We study the one-dimensional Lotka--Volterra competition--diffusion system
\begin{equation}\label{eq:original}
\begin{cases}
u_t=u_{xx}+u(1-u-av),\\
v_t=dv_{xx}+rv(1-v-bu),
\end{cases}\qquad t>0,\quad x\in\R.
\end{equation}
Here $u=u(t,x)$ and $v=v(t,x)$ denote the population densities of two competing species. We normalize both the diffusion coefficient and the intrinsic growth rate of the $u$-species to one. Thus, $d>0$ and $r>0$ are, respectively, the diffusion coefficient and intrinsic growth rate of the $v$-species relative to those of the $u$-species. The positive coefficients $a$ and $b$ measure the competitive effects exerted by the $v$-species on the $u$-species and by the $u$-species on the $v$-species, respectively.

\begin{bluechange}
The same system exhibits two distinct critical-parameter phenomena.
The first is a zero-speed threshold for bistable fronts in the strong-competition
regime $a>1$, $b>1$. The second is the threshold separating linearly and
nonlinearly selected minimal traveling-wave speeds in the monostable regime
$0<a<1$, $b>0$. The two regimes have different end states, initial-data
constructions, and notions of wave speed. Both integral representations are stated as separate main theorems in this
introduction: the bistable zero-speed threshold is denoted by $b_*$, while
the monostable linear-determinacy threshold is denoted by $b_0$.
The first theorem is proved in Sections~2--5 and the second in Sections~6--9.
Section~10 compares the two constructions. Neither threshold is
identified with the other.

The second problem relies on the preceding hand-derived theory of
Wu--Xiao--Zhou~\cite{WXXZ}. Their complete classification of minimal
traveling waves and decay rates, especially Theorems~2.13 and~2.16,
establishes the existence and uniqueness of the pulled-to-pushed transition
for all $0<a<1$ and $b>0$. The present paper does not replace their
classification by the Gaussian method; rather, it uses that classification to
compute an exact integral--limit representation of its critical coefficient.
\end{bluechange}

Before considering the interaction between the two species, we recall the corresponding scalar equations. If the $v$-species is absent, then $v\equiv0$ and the first equation of (1.1) becomes
\[
u_t=u_{xx}+u(1-u).
\]
Similarly, if the $u$-species is absent, then $u\equiv0$ and the second equation reduces to
\[
v_t=dv_{xx}+rv(1-v).
\]
These equations belong to the Fisher--KPP family introduced by Fisher~\cite{Fisher} and Kolmogorov, Petrovskii and Piskunov~\cite{KPP}. More generally, the scalar equation
\[
w_t=Dw_{xx}+Rw(1-w),\qquad D>0,\quad R>0,
\]
admits an increasing traveling profile
\[
w(t,x)=W(x+ct),\quad W(-\infty)=0,\quad W(+\infty)=1,
\]
if and only if $c\geq 2\sqrt{DR}$. Consequently, the minimal wave speeds of the uncoupled $u$- and $v$-equations are
\[
c_u=2,\qquad c_v=2\sqrt{dr},
\]
respectively. These scalar limits describe propagation when each species evolves in isolation. We now turn to the coupled system, for which the relevant front in the bistable regime connects the two species-exclusion states.

\subsection{The zero-speed threshold for bistable fronts}

\textcolor{black}{For the bistable zero-speed problem considered in Sections~1--5,} we assume the strong-competition conditions
\begin{equation}\label{eq:assumptions}
a>1,\qquad b>1,\qquad d>0,\qquad r>0.
\end{equation}
The equilibria $(0,1)$ and $(1,0)$ are then stable for the associated kinetic system, and the connecting front is of bistable type. We seek monotone traveling fronts of the form
\[
(u,v)(t,x)=(U,V)(\xi),\qquad \xi=x+ct,
\]
where $c\in\R$ is the propagation speed. The profiles satisfy
\begin{equation}\label{eq:TW}
\begin{cases}
U''-cU'+U(1-U-aV)=0,\\
dV''-cV'+rV(1-V-bU)=0,\\
(U,V)(-\infty)=(0,1),\quad (U,V)(+\infty)=(1,0),\\
U'>0>V'.
\end{cases}
\end{equation}
Here and below, primes denote differentiation with respect to $\xi$. With this convention, $c>0$ means that the $u$-species invades the region occupied by the $v$-species, whereas $c<0$ means that the $v$-species invades the region occupied by the $u$-species. The case $c=0$ corresponds to a standing front.

The existence and uniqueness, up to translation, of the monotone bistable front and the uniqueness of its speed are classical; see Gardner~\cite{Gardner}, Kan-on~\cite{KanOn}, and Nakamura and Ogiwara~\cite{NakamuraOgiwara}. We write this speed as $c=c(d,r,a,b)$. It depends continuously on the coefficients, is strictly decreasing in $a$, and is strictly increasing in $b$~\cite{KanOn,NakamuraOgiwara}. The problem of determining its sign has been studied through standing-wave identities, parameter comparison, and auxiliary competition systems. Sign criteria were obtained by Guo and Lin~\cite{GuoLin}, Chang, Chen and Wang~\cite{ChangChenWang}, and Xiao~\cite{Xiao}; further parameter regions were analyzed by Nakamura and Ogiwara~\cite{NakamuraOgiwara}. Ma, Xiao and Zhou~\cite{MaXiaoZhou} characterized the global zero-speed set by combining a Maxwell-type identity with rigidity of the standing phase curves. Related threshold and diffusion-dependence results are given by Kenne~\cite{Kenne}.

For fixed $a>1$ and $d,r>0$, let $b_*(a,d,r)>1$ denote the unique competition coefficient at which the front is stationary. The threshold theory and strict monotonicity in $b$ give
\begin{equation}\label{eq:sign}
\begin{aligned}
c(d,r,a,b)<0&\quad\Longleftrightarrow\quad 1<b<b_*(a,d,r),\\
c(d,r,a,b)=0&\quad\Longleftrightarrow\quad b=b_*(a,d,r),\\
c(d,r,a,b)>0&\quad\Longleftrightarrow\quad b>b_*(a,d,r).
\end{aligned}
\end{equation}
At zero speed, dividing the second equation of (1.3) by $r$ yields
\begin{equation}\label{eq:standing}
\begin{cases}
U''+U(1-U-aV)=0,\\
\rho V''+V(1-V-bU)=0,
\end{cases}\qquad \rho=\frac dr.
\end{equation}
Thus the standing-wave problem depends on $d$ and $r$ only through the ratio $\rho$, and the threshold can be written as
\begin{equation}\label{eq:B}
b_*(a,d,r)=B(a,\rho),\qquad \rho=\frac dr.
\end{equation}
\begin{bluechange}
Here $B(a,\rho)=b_*(a,\rho,1)$ denotes the same zero-speed threshold
in normalized coordinates.
\end{bluechange}
The threshold theorem of~\cite{MaXiaoZhou}, together with species exchange, also gives the explicit bounds
\begin{equation}\label{eq:Bbounds}
\sqrt a<B(a,\rho)<a^2\qquad(a>1,\ \rho>0).
\end{equation}
The precise background properties are recalled in Proposition~2.1.

This qualitative characterization leaves a representation question. The standing-wave identities retain a dependence on the unknown heteroclinic profile. For example, every monotone standing front satisfies the necessary Maxwell-type identity
\begin{equation}\label{eq:Maxwell}
3ab(\rho-1)E+(a-b)\bigl[\rho(a-1)+(b-1)\bigr]=0,
\end{equation}
where
\[
E=\int_{\R}(U+V-1)^2U'\dd\xi>0;
\]
see~\cite[Theorem~1.1]{MaXiaoZhou}. Since $E$ depends on the standing orbit, this identity alone is not an expression for the threshold in terms of the coefficients. The question addressed here is whether $B(a,\rho)$ can be represented by prescribed integral operations, without using either an unknown traveling-wave profile or the speed function as an input.

The purpose of this first part of the paper is to give such a representation for every $a>1$ and $\rho>0$. Our first main result expresses the threshold as a single limit of explicitly prescribed Gaussian integrals. The construction starts from a fixed logistic transition and alternates an explicit reaction step with heat convolution. Its excess mass detects the sign of the propagation speed, and integration with respect to the competition coefficient recovers the threshold. The qualitative threshold theory of~\cite{MaXiaoZhou} is used as a starting point; the contribution here is the integral representation and its approximation estimates.

We now specify the construction needed to state the first main result. In view of (1.6), we work with the auxiliary parameters $(d,r)=(\rho,1)$. The change of variables $(p,q)=(u,1-v)$ makes this auxiliary system cooperative. Let
\begin{equation}\label{eq:sG}
s(x)=\frac1{1+e^{-x}},\qquad G_{\kappa,h}(x)=\frac{e^{-x^2/(4\kappa h)}}{\sqrt{4\pi\kappa h}},\qquad \kappa,h>0,
\end{equation}
and set
\begin{equation}\label{eq:Ia}
I_a=[\sqrt a,a^2].
\end{equation}
For every integer $m\geq1$, define
\begin{equation}\label{eq:Nm}
N_m=\left\lceil(1+a^2)m^2e^{m^2}\right\rceil,\qquad h_m=\frac m{N_m}.
\end{equation}
Fix $\beta\in I_a$ and start from
\[
p_0^{m,\beta}=q_0^{m,\beta}=s.
\]
For $k=0,\ldots,N_m-1$, define
\begin{subequations}\label{eq:iter}
\begin{align}
p_{k+1}^{m,\beta}(x)&=\int_{\R}G_{1,h_m}(x-y)\left[p_k^{m,\beta}(y)+h_mF_{1,\beta}\bigl(p_k^{m,\beta}(y),q_k^{m,\beta}(y)\bigr)\right]\dd y,\label{eq:iter-p}\\
q_{k+1}^{m,\beta}(x)&=\int_{\R}G_{\rho,h_m}(x-y)\left[q_k^{m,\beta}(y)+h_mF_{2,\beta}\bigl(p_k^{m,\beta}(y),q_k^{m,\beta}(y)\bigr)\right]\dd y.\label{eq:iter-q}
\end{align}
\end{subequations}
where the two reaction polynomials are
\begin{equation}\label{eq:Fdef}
F_{1,\beta}(p,q)=p(1-a-p+aq),\qquad F_{2,\beta}(p,q)=(1-q)(\beta p-q).
\end{equation}
The first component does not depend explicitly on $\beta$; the subscript indicates that the two components belong to the same parameterized reaction vector. The dependence of the iterates on $a$ and $\rho$ is suppressed. Every finite iterate is specified by (1.9)--(1.13).

Define its excess mass by
\begin{equation}\label{eq:Mm}
M_m(a,\rho,\beta)=\int_{\R}\bigl[p_{N_m}^{m,\beta}(x)-s(x)\bigr]\dd x.
\end{equation}
We will prove that this integral is absolutely convergent and that it also admits the reaction-mass expression
\begin{equation}\label{eq:Mm2}
M_m(a,\rho,\beta)=h_m\sum_{k=0}^{N_m-1}\int_{\R}F_{1,\beta}\bigl(p_k^{m,\beta}(x),q_k^{m,\beta}(x)\bigr)\dd x.
\end{equation}
Thus the definition uses a prescribed initial function and a finite sequence of Gaussian convolutions.

\begin{theorem}\label{thm:main}
Let $a>1$ and $d,r>0$. For every $m\geq1$ and $\beta\in I_a$, the recursion (1.12) is well defined, its components take values in $[0,1]$, and the integrals in (1.14) and (1.15) are absolutely convergent and equal. Moreover,
\begin{equation}\label{eq:main}
\displaystyle b_*(a,d,r)=\sqrt a+\lim_{m\to\infty}\int_{\sqrt a}^{a^2}\frac{1}{1+\exp\bigl(M_m(a,d/r,\beta)\bigr)}\dd\beta.
\end{equation}
In particular, the right-hand side depends only on the prescribed coefficients, elementary functions, the Gaussian kernels, and the finite recursions (1.12).
\end{theorem}

(1.16) contains no unknown traveling-wave quantity: each finite approximation is determined by the coefficients, the prescribed initial function, and the Gaussian integral recursions. It is a single-limit representation. The number of Gaussian operations increases with $m$, and the step count in (1.11) is chosen to justify convergence.

The key point in the proof is to relate excess mass to front propagation. Let $\bfz(t,x)=(p(t,x),q(t,x))$ be the exact cooperative solution for $(d,r,b)=(\rho,1,\beta)$ with initial data $(s,s)$, and set
\[
J(t;a,\rho,\beta)=\int_{\R}[p(t,x)-s(x)]\dd x.
\]
Using front comparison and exponential spatial tail bounds, we prove
\[
\lim_{t\to\infty}\frac{J(t;a,\rho,\beta)}t=c(\rho,1,a,\beta).
\]
This step does not require convergence to a traveling wave in $L^1(\R)$.

We next compare the Gaussian recursion with the exact evolution. For a time step $h\leq(1+a^2)^{-1}$ and $T=Nh$, we establish the parameter-uniform estimate
\begin{equation}\label{eq:intro-error}
\norm{\bfZ_N-\bfz(T)}_{L^1(\R;\ell^1)}\leq C(a,\rho)\,Th\,e^{2\Lambda T},\qquad \Lambda=2a^2+1,
\end{equation}
where $\bfZ_N$ is the corresponding Gaussian iterate and $C(a,\rho)$ is explicit and independent of $\beta\in I_a$. With $T=m$ and $h=h_m$, the choice (1.11) makes this error tend to zero and therefore yields
\[
\lim_{m\to\infty}\frac{M_m(a,\rho,\beta)}m=c(\rho,1,a,\beta).
\]
The logistic factor in (1.16) consequently converges to the indicator of $[\sqrt a,B(a,\rho))$, except possibly at the single zero-speed parameter. Dominated convergence then gives the threshold formula. The estimate (1.17) controls the time-discrete approximation of the evolution; it does not provide a quantitative convergence rate for the final threshold approximation. This distinction is discussed in Remark~5.3.

\begin{bluechange}
\subsection{The linear-determinacy threshold in the monostable regime}

We return to system~\eqref{eq:original}, now under the assumptions
\begin{equation}\label{eq:LD-assumptions}
 0<a<1,\qquad d>0,\qquad r>0,\qquad b>0.
\end{equation}
Unlike the standing-front reduction in Subsection~1.1, the monostable
moving-front problem retains the diffusion and growth coefficients $d$ and $r$
separately; its speed does not reduce to a function of $d/r$ alone.
The finite mass $M_m$ defined below is a total mass for a compactly
supported initial function, not the signed excess mass $M_m$ of~(1.14).
We reuse the notation only within the present, disjoint parameter regime.
The linear candidate speed is $c_0=2\sqrt{1-a}$. Write $c^*(a,b,d,r)$
for the minimal propagation speed. Linear determinacy means $c^*=c_0$;
nonlinear determinacy means $c^*>c_0$.
Wu--Xiao--Zhou \cite[Theorem 2.13]{WXXZ} establish a unique threshold
$b_0(a,d,r)\in(0,\infty)$ such that
\begin{equation}\label{eq:LD-threshold}
\begin{aligned}
c^*(a,b,d,r)&=c_0&&\text{if }0<b\le b_0(a,d,r),\\
c^*(a,b,d,r)&>c_0&&\text{if }b>b_0(a,d,r).
\end{aligned}
\end{equation}
The parameter is denoted by $b^*$ in~\cite{WXXZ}; we use $b_0$
to distinguish it from the bistable zero-speed threshold $b_*$.
The formula below is a consequence of this theorem and the established
identification of spreading speed with minimal wave speed. We give the
additional approximation and integration arguments in full.

It is important to distinguish the established theorem from the new
representation. The existence and sharp characterization of
$b_0$ rely essentially on the traveling-wave classification of
Wu--Xiao--Zhou~\cite{WXXZ}, especially their Theorems~2.13 and~2.16.
Their slicing-method analysis identifies the pulled-to-pushed transition,
including the exceptional leading-edge decay of the transition front.
Our Gaussian construction does not independently prove this classification:
it converts its uniquely defined speed threshold into a prescribed
integral--limit expression without using a wave profile as input.


Fix $a,d,r$ as in \eqref{eq:LD-assumptions}. For the parameter $\beta\ge0$, set
\begin{equation}\label{eq:LD-A}
A(\beta)=\frac{1-a}{1-a\min\{\beta,1\}}
=\begin{cases}
 (1-a)/(1-a\beta),&0\le\beta<1,\\
 1,&\beta\ge1.
\end{cases}
\end{equation}
Introduce the elementary compactly supported function and the heat kernel
\begin{equation}\label{eq:LD-seed}
f_a(x)=
\begin{cases}
 \dfrac{1-a}{2}\cos^2\!\left(\dfrac{\pi x}{2}\right),&|x|\le1,\\[3pt]
 0,&|x|>1,
\end{cases}
\qquad
G_{\kappa,h}(x)=\frac{\exp(-x^2/(4\kappa h))}{\sqrt{4\pi\kappa h}}.
\end{equation}
For each integer $m\ge1$, define
\begin{equation}\label{eq:LD-steps}
K_\beta=2+r(1+\beta),\qquad
N_m(\beta)=\ceil{K_\beta m^2e^{m^2}},\qquad
h_m(\beta)=\frac{m}{N_m(\beta)}.
\end{equation}
For fixed $m,\beta$, abbreviate $N=N_m(\beta)$ and $h=h_m(\beta)$.
Starting with $p_0=f_a$ and $q_0=0$, perform $N$ steps:
\begin{equation}\label{eq:LD-recursion}
\begin{aligned}
p_{k+1}&=G_{1,h}*\bigl[p_k+h p_k(1-a-p_k+a q_k)\bigr],\\
q_{k+1}&=G_{d,h}*\bigl[q_k+hr(1-q_k)(\beta p_k-q_k)\bigr],
\qquad 0\le k<N.
\end{aligned}
\end{equation}
Here $(G*f)(x)=\int_\R G(x-y)f(y)\dd y$. The dependence of $p_k,q_k$
on $a,d,r,m,\beta$ is suppressed only to shorten the notation.
Define
\begin{align}
M_m(a,d,r;\beta)&=\int_\R p_{N_m(\beta)}(x)\dd x,
\label{eq:LD-mass}\\
D_m(a,d,r;\beta)&=\left(\frac{M_m(a,d,r;\beta)}{2A(\beta)m}
                      -2\sqrt{1-a}\right)_+,
\qquad s_+=\frac{s+|s|}{2}.
\label{eq:LD-gap}
\end{align}
Each $M_m$ and $D_m$ is completely specified by a finite number of Gaussian
integrations. In particular, no boundary-value problem or unknown traveling
profile occurs in their definition.

\begin{theorem}\label{thm:LD-main}
For every $0<a<1$ and $d,r>0$, the threshold in \eqref{eq:LD-threshold} satisfies
\begin{equation}\label{eq:LD-main}
\displaystyle
b_0(a,d,r)=
-\log\!\left[
 \lim_{n\to\infty}\lim_{m\to\infty}
 \int_0^\infty e^{-\beta}
       \bigl(1-e^{-nD_m(a,d,r;\beta)}\bigr)\dd\beta
\right].
\end{equation}
The inner limit is taken with $n$ fixed. The quantity inside the logarithm
equals $e^{-b_0}$ and belongs to $(0,1)$. For every $b>0$,
\begin{equation}\label{eq:LD-criterion}
\begin{aligned}
\text{linear determinacy}
 &\ \Longleftrightarrow\ \lim_{m\to\infty}D_m(a,d,r;b)=0,\\
\text{nonlinear determinacy}
 &\ \Longleftrightarrow\ \lim_{m\to\infty}D_m(a,d,r;b)>0.
\end{aligned}
\end{equation}
Moreover, the speed itself has the representation
\begin{equation}\label{eq:LD-speed}
c^*(a,b,d,r)=\lim_{m\to\infty}
\frac{M_m(a,d,r;b)}{2A(b)m}.
\end{equation}
\end{theorem}

\end{bluechange}

The remainder of the paper is organized as follows. Section~2 recalls the threshold reduction and establishes the relation between excess mass and propagation speed. Section~3 proves order preservation, spatial tail bounds, and the discrete mass identity for the Gaussian recursion. Section~4 derives the quantitative approximation estimate and justifies the simultaneous small-step and long-time limit. Section~5 proves Theorem~1.1 and discusses alternative step counts, the equal-ratio case, and the scope of the error estimate.
\begin{bluechange}
Sections~6--9 develop the monostable construction, establish its $L^1$
approximation and total-mass growth identities, and prove Theorem~\ref{thm:LD-main}.
Section~10 compares the scope of the two representations.
\end{bluechange}

\section{Bistable fronts and excess mass}
We first state explicitly the facts about the zero-speed threshold that are used below.

\begin{proposition}\label{prop:threshold}
For $a>1$ and $\rho>0$, there is a unique $B(a,\rho)\in(\sqrt a,a^2)$ such that, whenever $d/r=\rho$,
\begin{equation}\label{eq:threshold-sign}
\operatorname{sgn}c(d,r,a,b)=\operatorname{sgn}\bigl(b-B(a,\rho)\bigr),\qquad b>1.
\end{equation}
For fixed $a$, the map $\rho\mapsto B(a,\rho)$ is continuous and strictly increasing, with
\begin{equation}\label{eq:threshold-limits}
B(a,1)=a,\qquad \lim_{\rho\downarrow0}B(a,\rho)=\sqrt a,\qquad \lim_{\rho\to\infty}B(a,\rho)=a^2.
\end{equation}
The normalized speed
\begin{equation}\label{eq:cspeed}
c_{a,\rho}(\beta):=c(\rho,1,a,\beta)
\end{equation}
is continuous and strictly increasing in $\beta>1$, and its unique zero is $B(a,\rho)$.
\end{proposition}
\begin{proof}
In the notation of~\cite[Theorem~1.3]{MaXiaoZhou}, let $K_a(\sigma)$ be the zero-speed threshold in the first competition coefficient when the second competition coefficient equals $a$ and the diffusion-to-growth ratio is $\sigma$. That theorem states that $K_a$ is continuous and strictly decreasing, with range $(\sqrt a,a^2)$, central value $K_a(1)=a$, and limits $a^2$ at zero and $\sqrt a$ at infinity.

The species-exchange identity is
\begin{equation}\label{eq:exchange}
c(d,r,a,b)=-\sqrt{dr}\,c(1/d,1/r,b,a);
\end{equation}
see~\cite{MaXiaoZhou,NakamuraOgiwara}. The exchanged problem has ratio $\rho^{-1}$. Its speed is positive, zero, or negative according as $b$ is smaller than, equal to, or larger than $K_a(\rho^{-1})$. Therefore
\begin{equation}\label{eq:BKa}
B(a,\rho)=K_a(\rho^{-1}),
\end{equation}
and (2.1)--(2.2) follow. Continuity and strict increase of $c_{a,\rho}$ in $\beta$ are the classical reaction-parameter properties in~\cite{KanOn,NakamuraOgiwara}. The same threshold bounds also follow from~\cite[Corollary~2.9 and Proposition~2.13]{Kenne}.
\end{proof}

Fix $a>1$, $\rho>0$, and $\beta\in I_a$. Write
\[
\bfF_\beta=(F_{1,\beta},F_{2,\beta})^\top,\qquad
\bfD=\operatorname{diag}(1,\rho),\qquad \bfs=(s,s)^\top.
\]
The auxiliary initial-value problem is
\begin{equation}\label{eq:aux}
\bfz_t=\bfD\bfz_{xx}+\bfF_\beta(\bfz),\qquad \bfz(0,x)=\bfs(x),\qquad x\in\R.
\end{equation}
It is obtained from (1.1) with $(d,r,b)=(\rho,1,\beta)$ by setting $(z_1,z_2)=(u,1-v)$. The square $[0,1]^2$ is invariant, and the system is cooperative there because
\[
\partial_q F_{1,\beta}=ap\geq0,\qquad \partial_p F_{2,\beta}=\beta(1-q)\geq0.
\]
Standard parabolic theory and invariance give a unique global bounded classical solution. Comparison with its spatial translates shows that both components are nondecreasing in $x$.

\begin{lemma}\label{lem:transition}
For every $\eta\in(0,1/2)$, there exist $L_\eta,t_\eta>0$ such that the solution of (2.6) satisfies, for $t\geq t_\eta$,
\begin{equation}\label{eq:transition}
\begin{aligned}
z_1(t,x)&\leq\eta&&\text{if }x\leq-c_{a,\rho}(\beta)t-L_\eta,\\
z_1(t,x)&\geq1-\eta&&\text{if }x\geq-c_{a,\rho}(\beta)t+L_\eta.
\end{aligned}
\end{equation}
The constants may depend on $a$, $\rho$, $\beta$, and $\eta$.
\end{lemma}
\begin{proof}
We record the perturbed-front comparison construction used in bistable front stability; see~\cite{KanOnFang,RoquejoffreTermanVolpert}, \cite[Lemma~A2]{OgiwaraMatano}, and~\cite[proof of Corollary~2.2]{NakamuraOgiwara}. Fix the parameters and write $c=c_{a,\rho}(\beta)$. Let $\bfW=(U,1-V)^\top$ be the monotone front for these parameters. Thus
\[
\begin{gathered}
\bfD\bfW''-c\bfW'+\bfF_\beta(\bfW)=0,\\
\bfW(-\infty)=\zero,\qquad \bfW(+\infty)=\one,\\
W_i'>0\qquad(i=1,2),
\end{gathered}
\]
where $\zero=(0,0)^\top$ and $\one=(1,1)^\top$.

The reaction Jacobians at the two end states and suitable positive perturbation directions are
\[
D\bfF_\beta(\zero)=\begin{pmatrix}1-a&0\\\beta&-1\end{pmatrix},
\qquad \bfr_-=(1,\beta+1)^\top,
\]
\[
D\bfF_\beta(\one)=\begin{pmatrix}-1&a\\0&1-\beta\end{pmatrix},
\qquad \bfr_+=(a+1,1)^\top.
\]
In particular,
\[
D\bfF_\beta(\zero)\bfr_-=(1-a,-1)^\top<0,\qquad
D\bfF_\beta(\one)\bfr_+=(-1,1-\beta)^\top<0.
\]
Choose $\nu>0$ such that
\[
2\nu<\min\left\{a-1,\frac1{\beta+1},\frac1{a+1},\beta-1\right\}.
\]
By the limits of $\bfW$ and continuity of $D\bfF_\beta$, there is $R_0>0$ such that
\[
-D\bfF_\beta(\bfW(\xi))\bfr_- -\nu\bfr_-\geq\nu\bfr_-\quad(\xi\leq-R_0),
\]
\[
-D\bfF_\beta(\bfW(\xi))\bfr_+ -\nu\bfr_+\geq\nu\bfr_+\quad(\xi\geq R_0).
\]
All vector inequalities in this proof are componentwise. Take a positive $C^2$ function $\bfr:\R\to\R^2$ that equals $\bfr_-$ on $(-\infty,-R_0]$ and $\bfr_+$ on $[R_0,\infty)$; for example, use a smooth convex interpolation. Set
\[
r_{\min}=\inf_{\xi\in\R,\,i=1,2}r_i(\xi)>0,\qquad
r_{\max}=\sup_{\xi\in\R,\,i=1,2}r_i(\xi)<\infty.
\]
The derivatives $\bfr'$, $\bfr''$ have compact support. On $[-R_0,R_0]$, each $W_i'$ has a strictly positive minimum. Consequently, we can choose $\kappa>0$ so large that, for some $\gamma>0$,
\[
\begin{aligned}
\mathbf b(\xi):={}&\kappa\bfW'(\xi)-\nu\bfr(\xi)+c\bfr'(\xi)-\bfD\bfr''(\xi)\\
&-D\bfF_\beta(\bfW(\xi))\bfr(\xi)\geq\gamma\one\qquad(\xi\in\R).
\end{aligned}
\]
Indeed, the tail inequalities above give this bound outside the compact interval, while the term $\kappa\bfW'$ dominates its bounded remaining terms on that interval.

For a positive constant $\delta_0$ to be chosen, define
\[
\delta(t)=\delta_0e^{-\nu t},\qquad
\zeta_\pm(t)=\zeta_\pm^0\pm\frac{\kappa\delta_0}{\nu}(1-e^{-\nu t}),
\]
\[
\xi_\pm(t,x)=x+ct+\zeta_\pm(t),\qquad
\bfw^\pm(t,x)=\bfW(\xi_\pm(t,x))\pm\delta(t)\bfr(\xi_\pm(t,x)).
\]
Writing $\mathcal P(\bfw)=\bfw_t-\bfD\bfw_{xx}-\bfF_\beta(\bfw)$, direct differentiation gives
\[
\mathcal P(\bfw^\pm)=\pm\delta(t)\mathbf b(\xi_\pm)+\mathbf E_\pm(t,x),\qquad
\sup_x\norm{\mathbf E_\pm(t,x)}_\infty\leq C\delta(t)^2.
\]
Here $C\geq1$ is independent of $t$, $x$ and of the initial shifts $\zeta_\pm^0$. More explicitly, the remainder is
\[
\begin{aligned}
\mathbf E_\pm={}&\kappa\delta(t)^2\bfr'(\xi_\pm)\\
&-\bigl[\bfF_\beta(\bfW(\xi_\pm)\pm\delta(t)\bfr(\xi_\pm))
-\bfF_\beta(\bfW(\xi_\pm))\\
&\qquad\mp\delta(t)D\bfF_\beta(\bfW(\xi_\pm))\bfr(\xi_\pm)\bigr].
\end{aligned}
\]
The bound follows because $\bfF_\beta$ is quadratic and $\bfr$ and $\bfr'$ are bounded. Choose
\[
0<\delta_0<\min\left\{1,\frac\gamma{2C},\frac1{2r_{\max}}\right\}.
\]
It follows that $\mathcal P(\bfw^+)\geq\frac12\gamma\delta(t)\one$ and $\mathcal P(\bfw^-)\leq-\frac12\gamma\delta(t)\one$.

To keep the comparison inside the invariant square, truncate componentwise:
\[
\begin{aligned}
\overline w_i(t,x)&=\min\{1,w_i^+(t,x)\},\\
\underline w_i(t,x)&=\max\{0,w_i^-(t,x)\}\qquad(i=1,2).
\end{aligned}
\]
These are bounded generalized super- and subsolutions, respectively. To verify this, on any region where the $i$th component is not truncated, the explicit formulas for the reaction give
\[
F_{i,\beta}(\overline{\bfw})\leq F_{i,\beta}(\bfw^+),\qquad
F_{i,\beta}(\underline{\bfw})\geq F_{i,\beta}(\bfw^-).
\]
For $i=1$, this uses $aw_1^\pm\geq0$ on that region; for $i=2$, it uses $\beta(1-w_2^\pm)\geq0$ there. Thus the inequalities hold even if the other untruncated component lies outside $[0,1]$. On a region where a component is truncated to its boundary value, the required residual sign follows from
\[
\begin{aligned}
F_{1,\beta}(0,q)&=0,\qquad F_{1,\beta}(1,q)=a(q-1)\leq0,\\
F_{2,\beta}(p,0)&=\beta p\geq0,\qquad F_{2,\beta}(p,1)=0.
\end{aligned}
\]
Here $0\leq p,q\leq1$. The interface inequalities follow from the scalar truncation identities: for a smooth function $f$ and $\mathcal L_i=\partial_t-D_i\partial_{xx}$, where $D_1=1$ and $D_2=\rho$,
\[
\mathcal L_i\min\{1,f\}\geq\mathbf1_{\{f<1\}}\mathcal L_i f,\qquad
\mathcal L_i\max\{0,f\}\leq\mathbf1_{\{f>0\}}\mathcal L_i f
\]
in the sense of distributions. These follow by smooth concave and convex approximations of the two truncations, respectively, and include the favorable signs of their spatial derivative jumps. Together with the reaction inequalities above, they establish the generalized super- and subsolution inequalities. Both truncated functions take values in $[0,1]^2$, so the comparison principle uses only cooperativity on that square.

\begin{bluechange}
For clarity, these distributional assertions and the comparison on the whole line can be justified directly. Let $\phi_\varepsilon$ be a smooth concave approximation of $r\mapsto\min\{1,r\}$, with $0\leq\phi_\varepsilon'\leq1$. For a smooth $f$,
\[
\mathcal L_i\phi_\varepsilon(f)
=\phi_\varepsilon'(f)\mathcal L_i f-D_i\phi_\varepsilon''(f)|f_x|^2
\geq\phi_\varepsilon'(f)\mathcal L_i f.
\]
Passing to the limit locally in distributions gives the first truncation inequality. Smooth convex approximations of $r\mapsto\max\{0,r\}$ give the second. On the set $w_i^+<1$, the $i$th component of $\overline{\bfw}$ agrees with $w_i^+$, and the off-diagonal reaction derivative is nonnegative there; on the set $w_i^+\geq1$, the constant component $\overline w_i=1$ has $F_{i,\beta}(\overline{\bfw})\leq0$. Hence $\mathcal P_i(\overline{\bfw})\geq0$ in distributions. Similarly, on $w_i^->0$ the off-diagonal derivative is nonnegative, while on $w_i^-\leq0$ one has $F_{i,\beta}(\underline{\bfw})\geq0$. Thus $\mathcal P_i(\underline{\bfw})\leq0$ in distributions.

For completeness, we give the whole-line comparison argument, including
its treatment of the nonsmooth truncation. Let $v$ be a bounded
$[0,1]^2$-valued distributional subsolution and let $z$ be a bounded
$[0,1]^2$-valued classical solution, with $v(0,\cdot)\leq z(0,\cdot)$.
Put $w=v-z$. By the mean-value formula,
\[
F_{i,\beta}(v)-F_{i,\beta}(z)=\sum_{j=1}^2a_{ij}(t,x)w_j,
\qquad
a_{ij}=\int_0^1\partial_jF_{i,\beta}(z+\theta(v-z))\,\mathrm d\theta.
\]
The coefficients $a_{ij}$ are bounded, and $a_{ij}\geq0$ for $i\ne j$.
Writing $w_i^+=\max\{w_i,0\}$, the distributional Kato inequality
for $\partial_t-D_i\partial_{xx}$ gives
\[
(\partial_t-D_i\partial_{xx})w_i^+
\leq C_0(w_1^++w_2^+)\qquad (i=1,2),
\]
where $C_0$ depends only on the Lipschitz bound of $\bfF_\beta$ on the
square. To see this directly, apply smooth convex approximations of the
positive-part function to $w_i$ on $[\varepsilon,T]$, use the favorable sign
of their second derivatives, and then pass to the limit. The boundedness of
$w$ permits comparison with the positive scalar heat semigroups on $\R$.
With $E(t)=\max_i\sup_x w_i^+(t,x)$, their Gaussian integral formulas yield
\[
E(t)\leq 2C_0\int_0^t E(s)\,\mathrm ds.
\]
Here the initial term vanishes by the assumed initial ordering; the same
bound follows first on $[\varepsilon,T]$ and then as $\varepsilon\downarrow0$.
Gronwall's inequality gives $v\leq z$ throughout $[0,T]\times\R$.
The corresponding argument with $z-v$ gives comparison with bounded
supersolutions. Applying these two facts to the truncated functions above
propagates the stated ordering for all $t\geq0$.
\end{bluechange}

It remains to choose the initial shifts. Put $\mu=\delta_0r_{\min}\in(0,1/2)$ and $X_0=\log((1-\mu)/\mu)>0$. The prescribed logistic initial function satisfies $s(x)\leq\mu$ for $x\leq-X_0$ and $s(x)\geq1-\mu$ for $x\geq X_0$. Choose $\zeta_+^0$ large enough and $\zeta_-^0$ small enough that
\[
W_i(-X_0+\zeta_+^0)\geq1-\mu,\qquad
W_i(X_0+\zeta_-^0)\leq\mu\qquad(i=1,2).
\]
For $x\leq-X_0$, $w_i^+(0,x)\geq\mu\geq s(x)$; for $x\geq-X_0$, monotonicity of $W_i$ gives $w_i^+(0,x)\geq1$. Similarly, $w_i^-(0,x)\leq0$ for $x\leq X_0$, whereas $w_i^-(0,x)\leq1-\mu\leq s(x)$ for $x\geq X_0$. Hence
\[
\begin{aligned}
\underline{\bfw}(0,x)&\leq\bfs(x)\leq\overline{\bfw}(0,x),\\
\underline{\bfw}(t,x)&\leq\bfz(t,x)\leq\overline{\bfw}(t,x)\quad(t\geq0,\ x\in\R)
\end{aligned}
\]
by comparison.

Finally, define
\[
C_{\mathrm{sh}}=\max\{|\zeta_-^0|,|\zeta_+^0|\}+\frac{\kappa\delta_0}{\nu}.
\]
Then $|\zeta_\pm(t)|\leq C_{\mathrm{sh}}$ for all $t\geq0$. Given $\eta\in(0,1/2)$, choose $R_\eta>0$ such that $W_1(\xi)\leq\eta/2$ for $\xi\leq-R_\eta$ and $W_1(\xi)\geq1-\eta/2$ for $\xi\geq R_\eta$, and set
\[
L_\eta=R_\eta+C_{\mathrm{sh}}+1,\qquad
t_\eta=1+\frac1\nu\max\left\{0,\log\frac{2\delta_0r_{\max}}\eta\right\}.
\]
For $t\geq t_\eta$, the perturbation in either component is at most $\eta/2$. If $x\leq-ct-L_\eta$, then $\xi_+(t,x)\leq-R_\eta$ and the upper comparison gives $z_1(t,x)\leq\eta$. If $x\geq-ct+L_\eta$, then $\xi_-(t,x)\geq R_\eta$ and the lower comparison gives $z_1(t,x)\geq1-\eta$. This proves (2.7).
\end{proof}

We next establish spatial tail bounds and relate the excess mass to the propagation speed.
\begin{lemma}\label{lem:tails}
Set
\begin{equation}\label{eq:omega}
D_0=\max\{1,\rho\},\qquad\omega=D_0+a^2.
\end{equation}
For $\beta\in I_a$, $i=1,2$, $t\geq0$, and $x\in\R$, the solution of (2.6) satisfies
\begin{equation}\label{eq:tails}
\begin{aligned}
0\leq z_i(t,x)&\leq\min\{1,e^{x+\omega t}\},\\
0\leq 1-z_i(t,x)&\leq\min\{1,e^{-x+\omega t}\}.
\end{aligned}
\end{equation}
In particular, the excess mass
\begin{equation}\label{eq:J}
J(t;a,\rho,\beta)=\int_{\R}[z_1(t,x)-s(x)]\dd x
\end{equation}
is absolutely convergent for every finite $t\geq0$.
\end{lemma}
\begin{proof}
On the invariant square,
\[
F_{1,\beta}(p,q)\leq p,\qquad F_{2,\beta}(p,q)\leq\beta p.
\]
Equivalently, $\bfF_\beta(\bfz)\leq B_-\bfz$ componentwise, where
\[
B_-=\begin{pmatrix}1&0\\\beta&0\end{pmatrix}.
\]
Both diffusion coefficients are at most $D_0$, and both row sums of $B_-$ are at most $a^2$. Thus $\mathbf E_-(t,x)=e^{x+\omega t}\one$ satisfies
\[
\partial_t\mathbf E_- -\bfD\partial_{xx}\mathbf E_- -B_-\mathbf E_-
=e^{x+\omega t}\begin{pmatrix}\omega-2\\\omega-\rho-\beta\end{pmatrix}\geq0.
\]
Since $\bfs(x)\leq e^x\one$, comparison with this linear supersolution, followed by invariance, proves the first bound. The comparison on $\R$ can equivalently be obtained from the positive linear heat semigroup and its variation-of-constants formula; the Gaussian convolutions of the exponential envelope are finite.

\begin{bluechange}
To justify this comparison despite the unbounded exponential envelope, fix $M\geq1$ and put $g_M(x)=\min\{M,e^x\}$. The bounded function $g_M\one$ dominates $\bfs$. Let $T_-(t)$ be the positive semigroup on bounded functions generated by $\bfD\partial_{xx}+B_-$. Linear comparison gives
\[
\bfz(t,\cdot)\leq T_-(t)(g_M\one).
\]
As $M\to\infty$, positivity and monotone convergence in the Gaussian variation-of-constants series give, for every fixed $t$ and $x$,
\[
\bfz(t,x)\leq T_-(t)(e^x\one)
=e^x\exp\bigl(t(\bfD+B_-)\bigr)\one
\leq e^{x+\omega t}\one.
\]
The last inequality follows from $(\bfD+B_-)\one\leq\omega\one$ and positivity of the finite-dimensional matrix semigroup. In particular, no comparison principle with an unbounded initial datum is being assumed.
\end{bluechange}

For $\bfy=(1-z_1,1-z_2)$, direct calculation yields
\[
-F_{1,\beta}(1-y_1,1-y_2)\leq ay_2,\qquad
-F_{2,\beta}(1-y_1,1-y_2)\leq y_2
\]
when $\bfy\in[0,1]^2$. These inequalities read $-\bfF_\beta(\one-\bfy)\leq B_+\bfy$, where $\one=(1,1)^\top$ and
\[
B_+=\begin{pmatrix}0&a\\0&1\end{pmatrix}.
\]
Its row sums are again at most $a^2$. Therefore $\mathbf E_+(t,x)=e^{-x+\omega t}\one$ is a supersolution of the corresponding linear system. The initial inequality $\one-\bfs(x)\leq e^{-x}\one$ and linear comparison prove the second bound.

\begin{bluechange}
The same bounded-data argument applies to this right-tail estimate: use $\min\{M,e^{-x}\}\one$ as initial datum for the positive semigroup generated by $\bfD\partial_{xx}+B_+$, pass to the monotone limit, and use $(\bfD+B_+)\one\leq\omega\one$.
\end{bluechange}

For $x<0$, we have
\[
|z_1(t,x)-s(x)|\leq(e^{\omega t}+1)e^x,
\]
and for $x>0$ the analogous bound is $(e^{\omega t}+1)e^{-x}$. This proves absolute convergence in (2.10).
\end{proof}

\begin{lemma}\label{lem:mass-speed}
For every $\beta\in I_a$,
\begin{equation}\label{eq:mass-speed}
\lim_{t\to\infty}\frac{J(t;a,\rho,\beta)}t=c_{a,\rho}(\beta).
\end{equation}
\end{lemma}
\begin{proof}
Write $c=c_{a,\rho}(\beta)$ and let $H(x)=\mathbf1_{\{x>0\}}$. The identity $s(-x)=1-s(x)$ gives
\[
\int_{\R}(H-s)\dd x=0.
\]
Translation of the step function therefore yields
\begin{equation}\label{eq:stepmass}
\int_{\R}[H(x+ct)-s(x)]\dd x=ct.
\end{equation}
Choose $K>\max\{|c|,\omega\}+1$ and fix $\eta\in(0,1/2)$. For sufficiently large $t$, Lemma~2.2 gives
\[
\int_{-Kt}^{Kt}|z_1(t,x)-H(x+ct)|\dd x\leq2Kt\eta+2L_\eta.
\]
Outside $[-Kt,Kt]$, the step function is respectively zero and one. Lemma~2.3 thus bounds the sum of the two remaining tail integrals by $2e^{-(K-\omega)t}$. Subtracting (2.12) from (2.10) and dividing by $t$, we obtain
\[
\limsup_{t\to\infty}\left|\frac{J(t;a,\rho,\beta)}t-c\right|\leq2K\eta.
\]
Since $K$ is independent of $\eta$, letting $\eta\downarrow0$ proves (2.11).
\end{proof}

\section{Order-preserving Gaussian approximation}
For an arbitrary step $h>0$, define the diagonal heat semigroup and the explicit heat--reaction approximation by
\begin{equation}\label{eq:operators}
\begin{aligned}
\bfS_h\bf f&=(G_{1,h}*f_1,G_{\rho,h}*f_2)^\top,\\
\bfR_{h,\beta}(p,q)&=(p,q)^\top+h\bfF_\beta(p,q),\\
\bfQ_{h,\beta}\bf f&=\bfS_h\bigl(\bf f+h\bfF_\beta(\bf f)\bigr).
\end{aligned}
\end{equation}
Starting from $\bfZ_0=\bfs$, set
\[
\bfZ_{k+1}=\bfQ_{h,\beta}\bfZ_k,\qquad \bfZ_k=(p_k,q_k)^\top.
\]
With $h=h_m$ and $N=N_m$, this is exactly (1.12).

\begin{lemma}\label{lem:discrete}
Assume
\begin{equation}\label{eq:stepsize}
0<h\leq\frac1{1+a^2},\qquad \beta\in I_a.
\end{equation}
Then $\bfR_{h,\beta}$ preserves $[0,1]^2$ and is order preserving there. The Gaussian iterates $\bfZ_k$ take values in $[0,1]^2$, are componentwise nondecreasing in $x$, and satisfy
\begin{equation}\label{eq:discrete-tails}
\begin{aligned}
p_k(x),q_k(x)&\leq e^{x+\omega kh},\\
1-p_k(x),1-q_k(x)&\leq e^{-x+\omega kh}.
\end{aligned}
\end{equation}
For each finite $k$, both $\bfZ_k-\bfs$ and $\bfF_\beta(\bfZ_k)$ belong to $L^1(\R;\R^2)$.
\end{lemma}
\begin{proof}
The off-diagonal entries of $D\bfR_{h,\beta}$ are nonnegative on the square. Its diagonal entries satisfy
\[
\partial_p(\bfR_{h,\beta})_1\geq1-h(a+1)\geq0,\qquad
\partial_q(\bfR_{h,\beta})_2\geq1-h(\beta+1)\geq0,
\]
because $a\leq a^2$ and $\beta\leq a^2$. Hence the reaction map is order preserving. Since
\[
\bfR_{h,\beta}(0,0)=(0,0),\qquad \bfR_{h,\beta}(1,1)=(1,1),
\]
it also preserves the square. Gaussian convolution preserves order and bounds, so induction gives invariance for the iterates. Order preservation and translation invariance imply spatial monotonicity.

For the left-tail envelope, the reaction update satisfies
\[
p+hF_{1,\beta}(p,q)\leq(1+h)p,\qquad
q+hF_{2,\beta}(p,q)\leq q+h\beta p.
\]
For the right-tail envelope, put $\bfy=(1-p,1-q)^\top$. Then
\[
y_1-hF_{1,\beta}(1-y_1,1-y_2)\leq y_1+hay_2,
\]
\[
y_2-hF_{2,\beta}(1-y_1,1-y_2)\leq(1+h)y_2.
\]
Since $1,a,\beta\leq a^2$, each common envelope is multiplied by at most $1+ha^2$ during the reaction step. Gaussian convolution then gives the required induction for (3.3). In fact,
\[
G_{\kappa,h}*e^{\pm x}=e^{\kappa h}e^{\pm x},
\]
and one step multiplies the common exponential envelope by at most
\[
(1+ha^2)e^{D_0h}\leq e^{\omega h}.
\]
The induction starts from $s(x)\leq e^x$ and $1-s(x)\leq e^{-x}$. These tails give $\bfZ_k-\bfs\in L^1$. Since $\bfF_\beta$ vanishes at $(0,0)$ and $(1,1)$ and is Lipschitz on the square, the same tails give $\bfF_\beta(\bfZ_k)\in L^1$.
\end{proof}

\begin{lemma}\label{lem:mass-identity}
Under (3.2), for every positive integer $N$,
\begin{equation}\label{eq:discrete-mass}
\int_{\R}(p_N-s)\dd x=h\sum_{k=0}^{N-1}\int_{\R}F_{1,\beta}(p_k,q_k)\dd x.
\end{equation}
All integrals are absolutely convergent. For fixed $N$ and $h$, the left-hand side is continuous as a function of $\beta\in I_a$.
\end{lemma}
\begin{proof}
The reference function $s$ is not integrable, so preservation of excess mass under heat convolution requires a separate justification. For $y\in\R$, monotonicity of $s$ and $\int_\R s'\dd x=1$ give
\begin{equation}\label{eq:shift}
\int_{\R}|s(x-y)-s(x)|\dd x=|y|,\qquad
\int_{\R}[s(x-y)-s(x)]\dd x=-y.
\end{equation}
In particular, the double integral needed for Fubini's theorem is absolutely convergent, since
\[
\int_{\R}\int_{\R}G_{\kappa,h}(y)|s(x-y)-s(x)|\dd x\dd y
=\int_{\R}|y|G_{\kappa,h}(y)\dd y=2\sqrt{\frac{\kappa h}\pi}<\infty.
\]
Since $G_{\kappa,h}$ has zero mean, Fubini's theorem yields
\begin{equation}\label{eq:conv0}
\int_{\R}[G_{\kappa,h}*s-s]\dd x=0.
\end{equation}
For $f=s+w$ with $w\in L^1(\R)$, it follows that
\[
\int_{\R}[G_{\kappa,h}*f-s]\dd x=\int_{\R}w\dd x.
\]
Applying this equality to $f=p_k+hF_{1,\beta}(p_k,q_k)$ gives
\[
\int_{\R}(p_{k+1}-s)\dd x=\int_{\R}(p_k-s)\dd x+h\int_{\R}F_{1,\beta}(p_k,q_k)\dd x.
\]
Summation and $p_0=s$ prove (3.4). The integrability assertions follow from Lemma~3.1.

For fixed $N$ and $h$, the reaction update is continuous in the state and in $\beta$ and takes values in $[0,1]^2$. In each finite Gaussian convolution, its components are therefore dominated by the integrable Gaussian kernel. Induction and dominated convergence give pointwise continuity of $\bfZ_N$ in $\beta$. Writing $p_N^\beta$ to display this dependence, the uniform tails (3.3) give
\[
|p_N^\beta(x)-s(x)|\leq(e^{\omega Nh}+1)e^{-|x|},\qquad \beta\in I_a.
\]
This majorant belongs to $L^1(\R)$ and is independent of $\beta$. A second application of dominated convergence proves continuity of the excess-mass integral.
\end{proof}

The finite construction also preserves the order of the competition parameter. If $\beta_1\leq\beta_2$ and $(p,q)\in[0,1]^2$, then
\[
\bfR_{h,\beta_2}(p,q)-\bfR_{h,\beta_1}(p,q)
=\begin{pmatrix}0\\h(\beta_2-\beta_1)p(1-q)\end{pmatrix}\geq0.
\]
Combining this inequality with state-order preservation and positivity of Gaussian convolution, induction from the common initial function gives $\bfZ_k^{\beta_1}\leq\bfZ_k^{\beta_2}$. Since their differences from $\bfs$ are integrable, it follows that, for every prescribed $m$,
\[
M_m(a,\rho,\beta_1)\leq M_m(a,\rho,\beta_2).
\]
Thus the finite-$m$ integrand in (1.16) is nonincreasing in $\beta$. This order property supplements the continuity proved above; the limit argument does not require strict parameter monotonicity of the finite iterates.

\section{Quantitative \texorpdfstring{$L^1$}{L1} error estimates}
We use the norms
\begin{equation}\label{eq:norms}
\norm{\bf f}_1=\int_{\R}(|f_1|+|f_2|)\dd x,\qquad
\norm{\bf f}_\infty=\max_{i=1,2}\norm{f_i}_{L^\infty(\R)}.
\end{equation}
All constants below are independent of $\beta\in I_a$. In addition to $D_0$ in (2.8), set
\begin{equation}\label{eq:constants}
\begin{gathered}
\Lambda=2a^2+1,\qquad H_2=2a^2+2,\\
A=\frac{1+\rho}2+a+a^2-2,\qquad C_2=1+\frac{H_2}{2\Lambda},\\
C(a,\rho)=D_0\Lambda+\frac34 D_0H_2+\frac12\Lambda A.
\end{gathered}
\end{equation}
The constants are deliberately chosen to be explicit and uniform on the whole integration interval.

\begin{lemma}\label{lem:derivatives}
The solution of (2.6) satisfies, for $t\geq0$,
\begin{align}
\norm{\bfz(t)-\bfs}_1&\leq\frac A\Lambda(e^{\Lambda t}-1),\label{eq:4.3}\\
\norm{\bfz_x(t)}_1&\leq2e^{\Lambda t},\qquad \norm{\bfz_x(t)}_\infty\leq\frac14e^{\Lambda t},\label{eq:4.4}\\
\norm{\bfz_{xx}(t)}_1&\leq C_2e^{2\Lambda t},\label{eq:4.5}\\
\norm{\bfz_t(t)}_1&\leq(D_0C_2+A)e^{2\Lambda t},\label{eq:4.6}\\
\norm{\bfD\partial_{xx}\bfF_\beta(\bfz(t))}_1&\leq D_0\left(\Lambda C_2+\frac{H_2}2\right)e^{2\Lambda t}.\label{eq:4.7}
\end{align}
\end{lemma}
\begin{proof}
On $[0,1]^2$, the absolute values of the Jacobian entries are bounded by
\[
|D\bfF_\beta(p,q)|\leq\begin{pmatrix}a+1&a\\\beta&\beta+1\end{pmatrix}.
\]
Both the induced $\ell^1$ norm and the induced $\ell^\infty$ norm of this matrix are at most $\Lambda$. Moreover, the constant Hessian satisfies, for $v,w\in\R^2$,
\begin{equation}\label{eq:hessian}
|D^2\bfF_\beta(p,q)[v,w]|_{\ell^1}\leq H_2|v|_{\ell^\infty}|w|_{\ell^1}.
\end{equation}
Indeed, the two components on the left before taking absolute values are
\[
-2v_1w_1+a(v_1w_2+v_2w_1),\qquad -\beta(v_1w_2+v_2w_1)+2v_2w_2,
\]
and the resulting bound is $2+a+\beta\leq H_2$.

Fix a finite $T>0$. The diagonal heat semigroup $\bfS_t$ is a contraction in both norms in (4.1), and it is strongly continuous on $L^1(\R;\R^2)$ and on $BUC(\R;\R^2)$, the space of bounded uniformly continuous functions. The bounded mild solution belongs to $C([0,T];BUC(\R;\R^2))$, since its initial function is uniformly continuous and the reaction is locally Lipschitz on the invariant square. Although $\bfs\notin L^1$, the tails in Lemma~2.3 provide an integrable bound for $|\bfz(t)-\bfs|$ that is uniform for $0\leq t\leq T$. Dominated convergence therefore gives
\[
\mathbf w:=\bfz-\bfs\in C([0,T];L^1(\R;\R^2)).
\]
Furthermore, $s''\in L^1$, and componentwise heat convolution yields
\[
\bfS_t\bfs-\bfs=\int_0^t\bfS_\tau\bfD\bfs''\dd\tau
\quad\text{in }L^1(\R;\R^2).
\]
This identity also follows by differentiating the heat convolution of the smooth bounded function $s$, and then integrating in time; its right-hand side is an $L^1$ Bochner integral. Consequently, subtracting $\bfs$ from the bounded mild equation gives the following valid $L^1$ mild equation:
\[
\mathbf w(t)=\int_0^t\bfS_{t-\tau}\bigl[\bfD\bfs''+\bfF_\beta(\bfs)+\bfF_\beta(\bfz(\tau))-\bfF_\beta(\bfs)\bigr]\dd\tau.
\]
The logistic function satisfies
\[
\int_{\R}s(1-s)\dd x=1,\qquad\norm{s''}_{L^1(\R)}=\frac12.
\]
Since
\[
\bfF_\beta(\bfs)=\bigl((1-a)s(1-s),(\beta-1)s(1-s)\bigr)^\top,
\]
we have
\[
\norm{\bfD\bfs''+\bfF_\beta(\bfs)}_1\leq\frac{1+\rho}2+a+\beta-2\leq A.
\]
Semigroup contraction and Gronwall's inequality now prove (4.3). They also give
\begin{equation}\label{eq:reaction-L1}
\norm{\bfF_\beta(\bfz(t))}_1\leq\norm{\bfF_\beta(\bfs)}_1+\Lambda\norm{\bfz(t)-\bfs}_1\leq Ae^{\Lambda t}.
\end{equation}

We establish the spatial derivatives by difference quotients. For $\varepsilon\ne0$, set
\[
\delta_\varepsilon f(x)=\frac{f(x+\varepsilon)-f(x)}\varepsilon,
\]
\[
B_\varepsilon(t,x)=\int_0^1D\bfF_\beta\bigl((1-\theta)\bfz(t,x)+\theta\bfz(t,x+\varepsilon)\bigr)\dd\theta.
\]
Translation invariance of the heat semigroup and the mean-value formula give
\[
\delta_\varepsilon\bfz(t)=\bfS_t\delta_\varepsilon\bfs+
\int_0^t\bfS_{t-\tau}B_\varepsilon(\tau)\delta_\varepsilon\bfz(\tau)\dd\tau.
\]
This equation is legitimate in $L^1$: translations of $\mathbf w$ belong to $C([0,T];L^1)$, and $\delta_\varepsilon\bfs\in L^1$ because $s'\in L^1$. The square is convex, so both induced norms of $B_\varepsilon$ are at most $\Lambda$. Moreover,
\[
\norm{\delta_\varepsilon\bfs}_1\leq2,\qquad\norm{\delta_\varepsilon\bfs}_\infty\leq\frac14.
\]
Gronwall's inequality therefore bounds the difference quotients by $2e^{\Lambda t}$ in $L^1$ and by $e^{\Lambda t}/4$ in the uniform norm. To justify their convergence, let $\mathbf y$ be the unique solution, in $C([0,T];L^1)\cap C([0,T];BUC)$, of
\[
\mathbf y(t)=\bfS_t\bfs'+\int_0^t\bfS_{t-\tau}D\bfF_\beta(\bfz(\tau))\mathbf y(\tau)\dd\tau.
\]
Such a solution exists by the linear mild equation with bounded coefficients. The continuity of $\bfz$ with values in $BUC$ implies that $B_\varepsilon\to D\bfF_\beta(\bfz)$ uniformly in $(t,x)$ on $[0,T]\times\R$. Also, $\delta_\varepsilon\bfs\to\bfs'$ in $L^1$ and in the uniform norm. Subtracting the two mild equations and applying Gronwall's inequality therefore gives $\delta_\varepsilon\bfz\to\mathbf y$ in $C([0,T];L^1)$ and in $C([0,T];BUC)$. The limit is the spatial derivative $\bfz_x$ (equivalently, identify it first in distributions). Thus $\bfz_x=\mathbf y$, its initial value is $\bfs'$, and (4.4) follows, including at $t=0$.

For the second derivative, take a spatial difference quotient of this now justified mild equation for $\bfz_x$. Since $D^2\bfF_\beta$ is constant, the product difference is exactly
\[
\begin{aligned}
&\delta_\varepsilon\bigl[D\bfF_\beta(\bfz)\bfz_x\bigr](t,x)\\
&\quad=D\bfF_\beta(\bfz(t,x+\varepsilon))\delta_\varepsilon\bfz_x(t,x)\\
&\qquad+D^2\bfF_\beta[\delta_\varepsilon\bfz(t,x),\bfz_x(t,x)].
\end{aligned}
\]
The first coefficient again has induced $\ell^1$ norm at most $\Lambda$. By (4.8) and the first-derivative bounds, the additional source has $L^1$ norm at most
\[
H_2\norm{\delta_\varepsilon\bfz(t)}_\infty\norm{\bfz_x(t)}_1\leq\frac{H_2}2e^{2\Lambda t}.
\]
The initial difference quotient satisfies $\norm{\delta_\varepsilon\bfs'}_1\leq\norm{\bfs''}_1=1$. Semigroup contraction and Gronwall's inequality consequently give
\[
\norm{\delta_\varepsilon\bfz_x(t)}_1\leq e^{\Lambda t}+\frac{H_2}2\int_0^t e^{\Lambda(t-\tau)}e^{2\Lambda\tau}\dd\tau
\leq C_2e^{2\Lambda t}.
\]
For completeness, convergence of these difference quotients follows from the same linear mild-equation argument: their coefficients converge uniformly to $D\bfF_\beta(\bfz)$, their initial functions converge to $\bfs''$ in $L^1$ and in the uniform norm, and their sources converge in $C([0,T];L^1)\cap C([0,T];BUC)$ to $D^2\bfF_\beta[\bfz_x,\bfz_x]$.

\begin{bluechange}
Here are the details of the last convergence statement. Uniform continuity of the map $t\mapsto\bfz(t)$ in $BUC$, and compactness of its image for $t\in[0,T]$, imply uniform equicontinuity in $x$. Consequently,
\[
\sup_{0\leq t\leq T}\norm{D\bfF_\beta(\bfz(t,\cdot+\varepsilon))-D\bfF_\beta(\bfz(t,\cdot))}_\infty\longrightarrow0.
\]
The first-derivative argument already gives
\[
\delta_\varepsilon\bfz\longrightarrow\bfz_x
\quad\text{in }C([0,T];L^1)\cap C([0,T];BUC).
\]
Since $\sup_{0\leq t\leq T}(\norm{\bfz_x(t)}_1+\norm{\bfz_x(t)}_\infty)<\infty$, the bilinear estimate (4.8) yields
\[
\begin{aligned}
&\norm{D^2\bfF_\beta[\delta_\varepsilon\bfz(t),\bfz_x(t)]
-D^2\bfF_\beta[\bfz_x(t),\bfz_x(t)]}_1\\
&\qquad\leq H_2\norm{\delta_\varepsilon\bfz(t)-\bfz_x(t)}_\infty
\norm{\bfz_x(t)}_1.
\end{aligned}
\]
The right-hand side tends to zero uniformly for $0\leq t\leq T$.
The same argument, with the uniform norm in both factors, proves convergence in $C([0,T];BUC)$. Also, $\delta_\varepsilon\bfs'\to\bfs''$ in both spaces. Therefore, by subtracting the linear mild equations and applying Gronwall's inequality, $\delta_\varepsilon\bfz_x$ converges in $C([0,T];L^1)\cap C([0,T];BUC)$ to the unique solution of the limiting equation. Identifying this limit distributionally establishes the existence and asserted continuity of $\bfz_{xx}$, including at $t=0$.
\end{bluechange}

The limit is $\bfz_{xx}$ and satisfies
\[
\bfz_{xx}(t)=\bfS_t\bfs''+\int_0^t\bfS_{t-\tau}
\bigl[D\bfF_\beta(\bfz(\tau))\bfz_{xx}(\tau)+D^2\bfF_\beta[\bfz_x(\tau),\bfz_x(\tau)]\bigr]\dd\tau.
\]
In particular, $\bfz_{xx}\in C([0,T];L^1)$ with initial value $\bfs''$, and (4.5) follows.

Now $\bfD\bfz_{xx}+\bfF_\beta(\bfz)$ is continuous with values in $L^1$. The evolution equation, or its distributional time integral, gives
\[
\bfz(t)-\bfs=\int_0^t\bigl[\bfD\bfz_{xx}(\tau)+\bfF_\beta(\bfz(\tau))\bigr]\dd\tau
\quad\text{in }L^1.
\]
Hence $\bfz-\bfs\in C^1([0,T];L^1)$, with the derivative at zero understood as a right derivative. Combining (4.5) with (4.9) proves (4.6). Finally, the chain rule gives
\[
\partial_x\bfF_\beta(\bfz)=D\bfF_\beta(\bfz)\bfz_x,\qquad
\partial_{xx}\bfF_\beta(\bfz)=D\bfF_\beta(\bfz)\bfz_{xx}+D^2\bfF_\beta[\bfz_x,\bfz_x].
\]
Together with (4.9), these identities show that $\bfF_\beta(\bfz(t))\in W^{2,1}(\R;\R^2)$ for every $t\geq0$; in fact the reaction is continuous in this space on each finite time interval. Semigroup contraction, (4.8), and the previously proved bounds yield
\[
\norm{\bfD\partial_{xx}\bfF_\beta(\bfz(t))}_1\leq D_0\left(\Lambda C_2+\frac{H_2}2\right)e^{2\Lambda t},
\]
which is (4.7). Since $T$ was arbitrary, all conclusions hold for $t\geq0$.

Spatial monotonicity and the endpoint limits in Lemma~2.3 also give the sharper identity
\[
\norm{\bfz_x(t)}_1=\sum_{i=1}^2\int_{\R}\partial_x z_i(t,x)\dd x=2,\qquad t\geq0.
\]
The weaker first-derivative bound in (4.4) remains valid and is retained in the estimates above; no change to the explicit constants is needed.
\end{proof}

We now use these bounds to estimate the error of the Gaussian approximation.
\begin{proposition}\label{prop:approx}
Let $h$ satisfy (3.2), let $N\geq1$, and set $T=Nh$. For the Gaussian iterate $\bfZ_N$ and the exact solution of (2.6) with the same initial data,
\begin{equation}\label{eq:error}
\norm{\bfZ_N-\bfz(T)}_1\leq C(a,\rho)\,Th\,e^{2\Lambda T}.
\end{equation}
Consequently, for (1.11),
\begin{equation}\label{eq:diagonal-error}
\sup_{\beta\in I_a}|M_m(a,\rho,\beta)-J(m;a,\rho,\beta)|
\leq\frac{C(a,\rho)}{1+a^2}e^{-m^2+2\Lambda m}\longrightarrow0.
\end{equation}
\end{proposition}
\begin{proof}
The variation-of-constants formula and (3.1) give
\[
\bfz(t+h)-\bfQ_{h,\beta}\bfz(t)
=\int_0^h\bigl[\bfS_{h-\tau}\bfF_\beta(\bfz(t+\tau))
-\bfS_h\bfF_\beta(\bfz(t))\bigr]\dd\tau.
\]
Using $\bfS_h=\bfS_{h-\tau}\bfS_\tau$, split the integrand as
\[
\begin{aligned}
&\bfS_{h-\tau}\bfF_\beta(\bfz(t+\tau))-\bfS_h\bfF_\beta(\bfz(t))\\
&\qquad=\bfS_{h-\tau}\bigl[\bfF_\beta(\bfz(t+\tau))-\bfF_\beta(\bfz(t))\bigr]
+\bfS_{h-\tau}(I-\bfS_\tau)\bfF_\beta(\bfz(t)).
\end{aligned}
\]
The $L^1$ time differentiability established in Lemma~4.1 makes the time integral below a Bochner integral in $L^1$. To estimate the time variation, use Lipschitz continuity and (4.6) to obtain
\[
\begin{aligned}
\norm{\bfF_\beta(\bfz(t+\tau))-\bfF_\beta(\bfz(t))}_1
&\leq\Lambda\int_t^{t+\tau}\norm{\bfz_t(\sigma)}_1\dd\sigma\\
&\leq\Lambda\tau(D_0C_2+A)e^{2\Lambda(t+h)}.
\end{aligned}
\]
For the semigroup variation, note first that
\[
\partial_x\bfF_\beta(\bfz(t))=D\bfF_\beta(\bfz(t))\bfz_x(t)\in L^1,
\]
\[
\partial_{xx}\bfF_\beta(\bfz(t))=D\bfF_\beta(\bfz(t))\bfz_{xx}(t)+D^2\bfF_\beta[\bfz_x(t),\bfz_x(t)]\in L^1.
\]
Together with (4.9), this gives $\bfF_\beta(\bfz(t))\in W^{2,1}(\R;\R^2)$, the domain of the generator $\bfD\partial_{xx}$ of the diagonal heat semigroup on $L^1$. For every $\bf f$ in this domain, the identity
\[
(\bfS_\tau-I)\bf f=\int_0^\tau\bfS_\xi\bfD\bf f''\dd\xi
\]
and (4.7) imply
\[
\norm{(\bfS_\tau-I)\bfF_\beta(\bfz(t))}_1
\leq\tau D_0\left(\Lambda C_2+\frac{H_2}2\right)e^{2\Lambda t}.
\]
Splitting the local error into these two contributions and integrating in $\tau$ yields
\[
\begin{aligned}
\norm{\bfz(t+h)-\bfQ_{h,\beta}\bfz(t)}_1
&\leq\frac{h^2}2\left(2D_0\Lambda C_2+\Lambda A+\frac{D_0H_2}2\right)e^{2\Lambda(t+h)}\\
&=C(a,\rho)h^2e^{2\Lambda(t+h)}.
\end{aligned}
\]
The last equality follows from the definitions in (4.2).

For square-valued functions whose difference belongs to $L^1$, semigroup contraction gives
\[
\norm{\bfQ_{h,\beta}\bf f-\bfQ_{h,\beta}\bf g}_1\leq(1+h\Lambda)\norm{\bf f-\bf g}_1.
\]
Since the exact and discrete solutions remain in the invariant square, iteration of the local estimate gives
\[
\begin{aligned}
\norm{\bfZ_N-\bfz(T)}_1
&\leq C(a,\rho)h^2\sum_{k=0}^{N-1}(1+h\Lambda)^{N-1-k}e^{2\Lambda(k+1)h}\\
&\leq C(a,\rho)Nh^2e^{2\Lambda T}.
\end{aligned}
\]
Here we used $1+h\Lambda\leq e^{h\Lambda}$ and
\[
\Lambda(N-1-k)h+2\Lambda(k+1)h\leq2\Lambda Nh.
\]
This proves (4.10).

For the prescribed sequence, $h_m<(1+a^2)^{-1}$ for every $m\geq1$, and
\[
mh_m=\frac{m^2}{N_m}\leq\frac{e^{-m^2}}{1+a^2}.
\]
The difference between the first-component mass integrals is bounded by the vector $L^1$ error. Substituting $T=m$ and $h=h_m$ into (4.10) gives (4.11). Since $a$ is fixed, $e^{-m^2+2\Lambda m}$ tends to zero.
\end{proof}

\section{Proof of Theorem \ref{thm:main}}

\begin{proof}[Proof of Theorem \ref{thm:main}]
The step condition follows from (1.11). Lemmas~3.1 and~3.2 prove the well-definedness and invariance of every finite recursion, the absolute convergence of the mass integrals, and the equality of (1.14) and (1.15). They also show that $M_m(a,\rho,\cdot)$ is continuous on $I_a$.

By Lemma~2.4 and Proposition~4.2,
\begin{equation}\label{eq:mass-limit}
\lim_{m\to\infty}\frac{M_m(a,\rho,\beta)}m=c_{a,\rho}(\beta),\qquad \beta\in I_a.
\end{equation}
Proposition~2.1 shows that this limit is negative for $\beta<B(a,\rho)$ and positive for $\beta>B(a,\rho)$. Consequently,
\[
M_m(a,\rho,\beta)\longrightarrow
\begin{cases}-\infty,&\beta<B(a,\rho),\\+\infty,&\beta>B(a,\rho).
\end{cases}
\]
It follows that
\begin{equation}\label{eq:indicator}
\frac1{1+e^{M_m(a,\rho,\beta)}}\longrightarrow
\begin{cases}1,&\sqrt a\leq\beta<B(a,\rho),\\0,&B(a,\rho)<\beta\leq a^2.
\end{cases}
\end{equation}
No convergence of $M_m$ at the single parameter $\beta=B(a,\rho)$ is required. The integrands are bounded by one on the finite interval $I_a$. Dominated convergence and the strict bounds (1.7) give
\[
\lim_{m\to\infty}\int_{\sqrt a}^{a^2}\frac{1}{1+e^{M_m(a,\rho,\beta)}}\dd\beta
=\int_{\sqrt a}^{B(a,\rho)}1\dd\beta=B(a,\rho)-\sqrt a.
\]
Finally, (1.6) gives (1.16).
\end{proof}

\begin{remark}
The proof applies to any explicitly prescribed integer sequence $N_m$ satisfying
\begin{equation}\label{eq:alternative}
\frac m{N_m}\leq\frac1{1+a^2},\qquad \frac{m^2}{N_m}e^{2\Lambda m}\longrightarrow0.
\end{equation}
For example,
\begin{equation}\label{eq:example-N}
N_m=\left\lceil(1+a^2)m^2e^{(4a^2+3)m}\right\rceil
\end{equation}
also gives (1.16). In this case the right-hand side of the mass-error bound is at most
\[
\frac{C(a,\rho)}{1+a^2}e^{-m},
\]
since $2\Lambda=4a^2+2$. The conditions above guarantee convergence of the absolute mass error to zero. If only the threshold formula (1.16) is required, a weaker condition is sufficient:
\[
\frac m{N_m}\leq\frac1{1+a^2},\qquad \frac m{N_m}e^{2\Lambda m}\longrightarrow0.
\]
Indeed, (4.10) gives
\[
\sup_{\beta\in I_a}\frac{|M_m(a,\rho,\beta)-J(m;a,\rho,\beta)|}{m}
\leq C(a,\rho)\frac m{N_m}e^{2\Lambda m}\longrightarrow0.
\]
This is enough for (5.1) and hence for the indicator limit and the threshold formula. For example,
\[
N_m=\left\lceil(1+a^2)m^2e^{2\Lambda m}\right\rceil
\]
gives a normalized mass-error bound of $C(a,\rho)/[(1+a^2)m]$. For this choice the same estimate only bounds the absolute mass error by $C(a,\rho)/(1+a^2)$; it does not establish convergence of that absolute error to zero. We retain (1.11) for the stronger conclusion (4.11). The error estimate proved here does not justify a fixed time step or a polynomial step count as $m\to\infty$. This is a limitation of that estimate, not a nonconvergence assertion for every other choice of discretization.
\end{remark}

\begin{remark}
When $d=r$, Proposition~2.1 gives $b_*(a,d,r)=B(a,1)=a$. Accordingly,
\begin{equation}\label{eq:equal-ratio}
\lim_{m\to\infty}\int_{\sqrt a}^{a^2}\frac{1}{1+e^{M_m(a,1,\beta)}}\dd\beta=a-\sqrt a.
\end{equation}
This is consistent with species-exchange symmetry at $b=a$ and involves no perturbative restriction on $a-1$.
\end{remark}

\begin{remark}
The uniform bound (4.11) estimates the error between the exact excess mass and its Gaussian approximation at time $m$. It does not provide a convergence rate for (5.2) or a computable error bound for the final threshold approximation
\[
B_m(a,\rho):=\sqrt a+\int_{\sqrt a}^{a^2}\frac{1}{1+e^{M_m(a,\rho,\beta)}}\dd\beta.
\]
In particular, the proof uses only pointwise long-time convergence for each fixed $\beta$. A quantitative threshold error would require additional control near the zero-speed parameter. Errors due to spatial truncation and numerical quadrature are not included in (4.10), which treats exact Gaussian and spatial integrals.

The representation separates two roles. Bistable-front theory identifies a unique transition from negative to positive speed inside the explicit interval $I_a$. The Gaussian construction detects this transition through the sign of excess-mass growth and recovers its location by parameter integration. Although the number of prescribed operations is unbounded along the limiting sequence, each finite expression is completely determined by the coefficients and contains no unknown traveling-wave quantity.
\end{remark}

\begin{bluechange}
\section{Monostable evolution and equilibrium normalization}

Only for the proof, consider the initial-value problem
\begin{equation}\label{eq:LD-aux}
\LDz_t=\LDD\LDz_{xx}+F_\beta(\LDz),\qquad
\LDD=\diag(1,d),\qquad \LDz(0,x)=(f_a(x),0),
\end{equation}
where
\begin{equation}\label{eq:LD-F}
F_\beta(p,q)=
\begin{pmatrix}
 p(1-a-p+a q)\\
 r(1-q)(\beta p-q)
\end{pmatrix}.
\end{equation}
This is \eqref{eq:original} with $b=\beta$, $u=z_1$, and $v=1-z_2$, initially
$v\equiv1$.
In Sections~6--9, the reaction vector $F_\beta$ and the diffusion
matrix $\LDD$ are defined by the present $(d,r)$-dependent evolution;
they are distinct from the normalized $(\rho,1)$ reaction and diffusion
operators used in Sections~2--4. The off-diagonal derivatives of $F_\beta$ are nonnegative on
$[0,1]^2$:
\[
\partial_q F_{\beta,1}=ap\ge0,\qquad
\partial_p F_{\beta,2}=r\beta(1-q)\ge0.
\]
Thus the usual comparison principle applies there.

Let
\begin{equation}\label{eq:LD-E}
Q(\beta)=A(\beta)\min\{\beta,1\},\qquad
E_\beta=(A(\beta),Q(\beta)).
\end{equation}
Direct substitution gives $F_\beta(0,0)=F_\beta(E_\beta)=0$.
The initial vector lies between these two constant equilibria. Consequently,
\begin{equation}\label{eq:LD-invariant}
0\le z_1(t,x)\le A(\beta)\le1,\qquad
0\le z_2(t,x)\le Q(\beta)\le1.
\end{equation}
\Needspace{8\baselineskip}
The invading equilibrium in the original variables is
\begin{center}
\begin{tabular}{@{}ll@{}}
\toprule
Parameter range & Invading equilibrium $(u,v)$\\
\midrule
$0<b<1$ & $\bigl((1-a)/(1-ab),(1-b)/(1-ab)\bigr)$\\[2pt]
$b=1$ & $(1,0)$, with a degenerate resident direction\\
$b>1$ & $(1,0)$\\
\bottomrule
\end{tabular}
\end{center}
In particular, the factor $A(b)$ in \eqref{eq:LD-gap} and \eqref{eq:LD-speed}
accounts for the coexistence equilibrium when $b<1$.


For $\beta>0$, $\beta\ne1$, we use the established spreading result
for \eqref{eq:original} with a nonzero, nonnegative compactly supported
$u$-component and a resident $v$-component bounded away from zero;
see \cite{LLW,LWL} and the discussion following (1.14) in \cite{WXXZ}.
Its $u$-invasion speed equals the minimal traveling-wave speed
$c^*(a,\beta,d,r)$. The initial datum used here is admissible:
$f_a$ is continuous, compactly supported, strictly positive on $(-1,1)$,
and the original resident datum is identically one. In particular, the
spreading statement is used in its usual interior-convergence and
exterior-decay forms, not merely as an estimate on the location of a level set.
For the specific initial data in \eqref{eq:LD-aux}, the cited results give
\begin{align}
\sup_{|x|\le ct}|z_1(t,x;\beta)-A(\beta)|&\longrightarrow0
&&\text{for }0<c<c^*(a,\beta,d,r),\label{eq:LD-inside}\\
\sup_{|x|\ge ct}z_1(t,x;\beta)&\longrightarrow0
&&\text{for }c>c^*(a,\beta,d,r).\label{eq:LD-outside}
\end{align}
We also use the continuity of $\beta\mapsto c^*(a,\beta,d,r)$ on
$(0,\infty)$, as recalled and treated in \cite[Sections 5--6]{WXXZ}.

For completeness, these facts imply \eqref{eq:LD-inside}--\eqref{eq:LD-outside}
at $\beta=1$ without requiring an exponential stability estimate at its
degenerate equilibrium. The initial data in \eqref{eq:LD-aux} do not depend
on $\beta$, and $F_\beta$ is componentwise nondecreasing in $\beta$.
Comparison therefore gives
\begin{equation}\label{eq:LD-parameter-order}
\LDz(t,x;\beta_-)\le\LDz(t,x;1)\le\LDz(t,x;\beta_+)
\quad\text{if }0<\beta_-<1<\beta_+.
\end{equation}
Fix $0<c<c^*(a,1,d,r)$. By continuity, $c<c^*(a,\beta_-,d,r)$ for all
$\beta_-<1$ sufficiently close to $1$. Applying \eqref{eq:LD-inside} to
the lower solution and using $z_1(t,x;1)\le1$ yields
\[
\limsup_{t\to\infty}\sup_{|x|\le ct}(1-z_1(t,x;1))
\le1-A(\beta_-).
\]
Letting $\beta_-\uparrow1$ proves \eqref{eq:LD-inside} at $1$.
For $c>c^*(a,1,d,r)$, choose $\beta_+>1$ sufficiently close to $1$ that
$c>c^*(a,\beta_+,d,r)$. The upper comparison in
\eqref{eq:LD-parameter-order} and \eqref{eq:LD-outside} prove the exterior
statement at $1$. At $\beta=0$, $z_2\equiv0$ and the first equation is
the scalar logistic equation with growth $1-a$; its speed is $c_0$.

\section{Recovering the minimal monostable speed from total mass}

\begin{lemma}\label{lem:LD-mass}
For every fixed $\beta>0$, the solution of \eqref{eq:LD-aux} satisfies
\begin{equation}\label{eq:LD-continuous-mass}
\lim_{t\to\infty}\frac1t\int_\R z_1(t,x;\beta)\dd x
=2A(\beta)c^*(a,\beta,d,r).
\end{equation}
\end{lemma}

\begin{proof}
Write $A=A(\beta)$, $c^*=c^*(a,\beta,d,r)$, and
$\sigma=(1-a)/2=\int_\R f_a$. Since $F_{\beta,1}(p,q)\le p$ on
$[0,1]^2$, scalar comparison gives
\begin{equation}\label{eq:LD-heat-bound}
z_1(t,x)\le e^t(G_{1,t}*f_a)(x).
\end{equation}
This proves integrability at every finite time. For $K>2$ and $Kt>1$,
the support of $f_a$ and the Gaussian tail bound imply
\begin{align}
\int_{|x|\ge Kt}z_1(t,x)\dd x
&\le \sigma e^t\int_{|y|\ge Kt-1}G_{1,t}(y)\dd y\notag\\
&\le2\sigma\exp\!\left(t-\frac{(Kt-1)^2}{4t}\right)
\longrightarrow0.\label{eq:LD-tail}
\end{align}
Indeed, for a Gaussian of variance $2t$, applying the exponential moment
bound separately on the two tails gives
$\int_{|y|\ge R}G_{1,t}(y)\dd y\le2e^{-R^2/(4t)}$.

For $0<c<c^*$, \eqref{eq:LD-inside} gives
\[
\liminf_{t\to\infty}\frac1t\int_\R z_1(t,x)\dd x\ge2cA.
\]
Let $c\uparrow c^*$ to obtain the required lower bound.
For the upper bound, fix $c>c^*$ and then $K>\max\{2,c\}$. Split the
integral into $|x|\le ct$, $ct<|x|<Kt$, and $|x|\ge Kt$.
By \eqref{eq:LD-invariant}, the first part is at most $2cAt$.
By \eqref{eq:LD-outside}, the second part is at most
$2Kt\sup_{|x|\ge ct}z_1(t,x)=o(t)$. The final part tends to zero by
\eqref{eq:LD-tail}. Thus the upper limit is at most $2cA$.
Letting $c\downarrow c^*$ proves \eqref{eq:LD-continuous-mass}.
\end{proof}

\section{Gaussian approximation for the monostable problem}

All estimates in this section are for fixed $a,d,r,\beta$. No estimate
uniform over the unbounded $\beta$ interval will be needed.
For a vector-valued function $w=(w_1,w_2)$, use
\[
\norm{w}_1=\int_\R(|w_1|+|w_2|),\qquad
\norm{w}_\infty=\max_{i=1,2}\norm{w_i}_{L^\infty(\R)}.
\]
Define the diagonal heat semigroup
\[
S_t w=(G_{1,t}*w_1,G_{d,t}*w_2).
\]
It is a contraction in both norms.


Let $T_h(z)=z+hF_\beta(z)$ and $\mathcal Q_h=S_hT_h$.
If $h\le K_\beta^{-1}$, the entries of $DT_h$ on $[0,1]^2$ satisfy
\begin{align*}
\partial_p(T_h)_1&\ge1-h(1+a)>0,&
\partial_q(T_h)_1&=hap\ge0,\\
\partial_p(T_h)_2&=hr\beta(1-q)\ge0,&
\partial_q(T_h)_2&\ge1-hr(1+\beta)>0.
\end{align*}
Thus $T_h$ is order preserving. Since $T_h(0)=0$ and
$T_h(E_\beta)=E_\beta$, the interval $[0,E_\beta]$ is preserved by
$T_h$ and by $\mathcal Q_h$. The choice \eqref{eq:LD-steps} satisfies the
step restriction, so every iterate in \eqref{eq:LD-recursion} obeys the
bounds \eqref{eq:LD-invariant}.

Each iterate is integrable: initially this is immediate, and on the cube
\[
0\le(T_h)_1\le(1+h)p,\qquad
0\le(T_h)_2\le q+hr\beta p.
\]
The heat semigroup preserves integrals of nonnegative integrable functions.
In particular, $M_m$ is finite and also has the exact finite-sum expression
\begin{equation}\label{eq:LD-mass-sum}
M_m=\sigma+h\sum_{k=0}^{N-1}
\int_\R p_k(1-a-p_k+a q_k)\dd x.
\end{equation}

We next derive the regularity estimates needed to control the
approximation error.

Set
\begin{equation}\label{eq:LD-constants}
L=3+r(1+2\beta),\qquad
H=2+2a+2r(1+\beta),\qquad d_0=\max\{1,d\}.
\end{equation}
On the cube, both the row and column sums of the absolute Jacobian are
bounded by $L$. Hence, pointwise and after integration,
\begin{equation}\label{eq:LD-derivative-bounds}
\begin{aligned}
\norm{DF_\beta(z)w}_j&\le L\norm{w}_j\quad(j=1,\infty),\\
\norm{D^2F_\beta(z)[v,w]}_1&\le H\norm{v}_\infty\norm{w}_1.
\end{aligned}
\end{equation}
For clarity, the two Hessian components are
\[
-2v_1w_1+a(v_1w_2+v_2w_1),\qquad
r\{-\beta(v_1w_2+v_2w_1)+2v_2w_2\}.
\]
These formulas directly verify the bound with $H$ above.

The initial function belongs to $W^{2,1}(\R)\cap W^{2,\infty}(\R)$.
Its first derivative vanishes at $x=\pm1$, so its distributional second
derivative has no point masses there. Its relevant norms are
\begin{equation}\label{eq:LD-seed-norms}
\begin{aligned}
B_0&=\norm{\LDz(0)}_1=\sigma,
&B_1&=\norm{\LDz_x(0)}_1=2\sigma,\\
B_\infty&=\norm{\LDz_x(0)}_\infty=\frac{\pi\sigma}{2},
&B_2&=\norm{\LDz_{xx}(0)}_1=2\pi\sigma.
\end{aligned}
\end{equation}
From Duhamel's formula, $F_\beta(0)=0$, and
\eqref{eq:LD-derivative-bounds}, Gronwall's inequality gives
\begin{equation}\label{eq:LD-z-norms}
\norm{\LDz(t)}_1\le B_0e^{Lt},\qquad
\norm{\LDz_x(t)}_1\le B_1e^{Lt},\qquad
\norm{\LDz_x(t)}_\infty\le B_\infty e^{Lt}.
\end{equation}
For the second derivative, differentiating the reaction term twice yields
\[
\partial_{xx}F_\beta(\LDz)
=DF_\beta(\LDz)\LDz_{xx}+D^2F_\beta(\LDz)[\LDz_x,\LDz_x].
\]
Duhamel's formula and \eqref{eq:LD-z-norms} therefore give
\begin{align*}
\norm{\LDz_{xx}(t)}_1
&\le e^{Lt}B_2+
H B_\infty B_1 e^{Lt}\int_0^t e^{Ls}\dd s\\
&\le \widehat B_2 e^{2Lt},\qquad
\widehat B_2=B_2+\frac{H B_\infty B_1}{L}.
\end{align*}
To justify the derivative estimates at the initial time, one may use
spatial difference quotients in the $L^1$ mild equation. Since $f_a$ is
compactly supported and belongs to $W^{2,1}(\R)\cap W^{1,\infty}(\R)$,
its first difference quotients converge in $L^1$ and uniformly to $f_a'$;
indeed, $f_a'$ is uniformly continuous. Writing
$\delta_\varepsilon f(x)=[f(x+\varepsilon)-f(x)]/\varepsilon$, the
reaction difference is
\[
F_\beta(\LDz(t,x+\varepsilon))-F_\beta(\LDz(t,x))
=B_\varepsilon(t,x)\delta_\varepsilon\LDz(t,x),
\]
where $\norm{B_\varepsilon}_{\ell^j\to\ell^j}\le L$ for
$j=1,\infty$. Gronwall's inequality first yields the bounds for
$\delta_\varepsilon\LDz$, and their limit is $\LDz_x$ in
$C([0,T];L^1)\cap C([0,T];BUC)$. Its mild equation has coefficient
$DF_\beta(\LDz)$ and initial value $(f_a',0)$.

Apply $\delta_\varepsilon$ to this first-derivative equation. The identity
\[
\begin{aligned}
\delta_\varepsilon[DF_\beta(\LDz)\LDz_x]
={}&DF_\beta(\LDz(\cdot+\varepsilon))
\delta_\varepsilon\LDz_x\\
&+D^2F_\beta[\delta_\varepsilon\LDz,\LDz_x]
\end{aligned}
\]
has its second term bounded in $L^1$ by
$H\norm{\delta_\varepsilon\LDz}_\infty\norm{\LDz_x}_1$.
Moreover, $\delta_\varepsilon f_a'\to f_a''$ in $L^1$,
$DF_\beta(\LDz(\cdot+\varepsilon))\to DF_\beta(\LDz)$ uniformly,
and the source converges in $C([0,T];L^1)$ by the established convergence
of $\delta_\varepsilon\LDz$ and the bound on $\LDz_x$.
Subtracting the corresponding mild equations and using Gronwall's
inequality therefore shows that $\delta_\varepsilon\LDz_x$ converges
in $C([0,T];L^1)$ to $\LDz_{xx}$, with initial value $(f_a'',0)$.
No uniform convergence of $\delta_\varepsilon f_a'$ is required: $f_a''$
is allowed to have jump discontinuities at $x=\pm1$.
To make the last convergence quantitative, denote by $y$ the solution of
the limiting linear mild equation for $\LDz_{xx}$ and put
$e_\varepsilon=\delta_\varepsilon\LDz_x-y$.
Write $A_\varepsilon=DF_\beta(\LDz(\cdot+\varepsilon))$,
$A=DF_\beta(\LDz)$, and
$H_\varepsilon=D^2F_\beta[\delta_\varepsilon\LDz,\LDz_x]$,
$H_0=D^2F_\beta[\LDz_x,\LDz_x]$. Define
\[
\begin{aligned}
\eta_\varepsilon(T)={}&\|\delta_\varepsilon f_a'-f_a''\|_{L^1}
+T\sup_{0\le t\le T}\|A_\varepsilon(t)-A(t)\|_\infty
  \sup_{0\le t\le T}\|y(t)\|_1\\
&+T\sup_{0\le t\le T}\|H_\varepsilon(t)-H_0(t)\|_1.
\end{aligned}
\]
The established convergence in $C([0,T];L^1)$ and in the uniform
norm gives $\eta_\varepsilon(T)\to0$. Subtracting the two mild
equations, and using $\|A_\varepsilon\|_{\ell^1\to\ell^1}\le L$, yields
\[
\sup_{0\leq t\leq T}\|e_\varepsilon(t)\|_1
\leq \eta_\varepsilon(T)+L\int_0^T\|e_\varepsilon(s)\|_1\,\mathrm ds
\leq e^{LT}\eta_\varepsilon(T).
\]
Identifying $y$ as the distributional second spatial derivative yields
$\LDz_{xx}\in C([0,T];L^1)$. In particular, the preceding inequalities
and the following chain rules hold in $L^1$ also at $t=0$. Consequently,
\begin{align}
\norm{\LDz_t(t)}_1&\le B_t e^{2Lt},&
B_t&=d_0\widehat B_2+LB_0,\label{eq:LD-zt}\\
\norm{\LDD\partial_{xx}F_\beta(\LDz(t))}_1&\le B_F e^{2Lt},&
B_F&=d_0(L\widehat B_2+H B_\infty B_1).
\label{eq:LD-Fxx}
\end{align}

These estimates yield the following bound for the Gaussian
approximation.

\begin{lemma}\label{lem:LD-error}
Let $\LDZ_{k+1}=\mathcal Q_h\LDZ_k$, $\LDZ_0=\LDz(0)$, and $T=Nh$.
When $h\le K_\beta^{-1}$,
\begin{equation}\label{eq:LD-global-error}
\norm{\LDZ_N-\LDz(T)}_1\le C_{\mathrm{loc}}T h e^{2LT},\qquad
C_{\mathrm{loc}}=\tfrac12(LB_t+B_F).
\end{equation}
\end{lemma}

\begin{proof}
For an exact state at time $t$, Duhamel's formula gives
\begin{align*}
\LDz(t+h)-\mathcal Q_h\LDz(t)
=\int_0^h\bigl[S_{h-s}F_\beta(\LDz(t+s))
                  -S_hF_\beta(\LDz(t))\bigr]\dd s.
\end{align*}
Split the integrand into
\[
S_{h-s}\bigl(F_\beta(\LDz(t+s))-F_\beta(\LDz(t))\bigr)
 +(S_{h-s}-S_h)F_\beta(\LDz(t)).
\]
The first term has $L^1$ norm at most
$LB_t s e^{2L(t+h)}$ by \eqref{eq:LD-zt}.
For the second term use $S_h=S_{h-s}S_s$ and
\[
\norm{(I-S_s)w}_1
\le s\norm{\LDD w_{xx}}_1
\quad\text{for }w\in[W^{2,1}(\R)]^2.
\]
This inequality follows by integrating
$\partial_\tau S_\tau w=S_\tau\LDD w_{xx}$ from $0$ to $s$.
By \eqref{eq:LD-Fxx}, the second term has norm at most
$B_F s e^{2Lt}$. Integration in $s$ now yields
\begin{equation}\label{eq:LD-local-error}
\norm{\LDz(t+h)-\mathcal Q_h\LDz(t)}_1
\le C_{\mathrm{loc}}h^2e^{2L(t+h)}.
\end{equation}
For two states taking values in the cube,
\[
\norm{\mathcal Q_h w-\mathcal Q_h\widetilde w}_1
\le(1+hL)\norm{w-\widetilde w}_1.
\]
Iterating this estimate with \eqref{eq:LD-local-error} gives
\begin{align*}
\norm{\LDZ_N-\LDz(T)}_1
&\le C_{\mathrm{loc}}h^2
\sum_{j=0}^{N-1}(1+hL)^{N-1-j}e^{2L(j+1)h}\\
&\le C_{\mathrm{loc}}N h^2e^{2LT}
=C_{\mathrm{loc}}T h e^{2LT}.
\end{align*}
\end{proof}

With $T=m$ and the choice \eqref{eq:LD-steps}, the estimate becomes
\begin{equation}\label{eq:LD-diagonal}
\left|M_m-\int_\R z_1(m,x;\beta)\dd x\right|
\le \frac{C_{\mathrm{loc}}}{K_\beta}
\exp(-m^2+2Lm)\longrightarrow0.
\end{equation}
Combining \eqref{eq:LD-diagonal} with Lemma~\ref{lem:LD-mass} proves
\eqref{eq:LD-speed}. Continuity of the positive-part function then gives
\begin{equation}\label{eq:LD-gap-limit}
D_m(a,d,r;\beta)\longrightarrow
\Delta(\beta):=c^*(a,\beta,d,r)-c_0\ge0
\quad\text{for every }\beta>0.
\end{equation}

\section{Proof of Theorem \ref{thm:LD-main}}

For $n,m\ge1$, let
\begin{equation}\label{eq:LD-I}
I_{n,m}=\int_0^\infty e^{-\beta}
              \bigl(1-e^{-nD_m(a,d,r;\beta)}\bigr)\dd\beta.
\end{equation}
The integrand is measurable. Indeed, $N_m(\beta)$ is integer valued and
constant on Borel intervals. On each such interval, $h=m/N_m$ is fixed,
and the finite reaction and heat operations depend continuously on
$\beta$ in $L^1$ on compact parameter subintervals. Their integrals are
therefore Borel functions. Furthermore,
\begin{equation}\label{eq:LD-domination}
0\le e^{-\beta}(1-e^{-nD_m})\le e^{-\beta},\qquad
\int_0^\infty e^{-\beta}\dd\beta=1.
\end{equation}
Thus no estimate uniform in $\beta$ is needed for the splitting error.

\begin{proof}[Proof of Theorem~\ref{thm:LD-main}]
The speed representation and \eqref{eq:LD-criterion} follow from
\eqref{eq:LD-speed} and \eqref{eq:LD-gap-limit}. To obtain the threshold,
fix $n$. By \eqref{eq:LD-gap-limit}, \eqref{eq:LD-domination}, and dominated
convergence,
\begin{equation}\label{eq:LD-inner-limit}
\lim_{m\to\infty}I_{n,m}
=\int_0^\infty e^{-\beta}(1-e^{-n\Delta(\beta)})\dd\beta.
\end{equation}
Write $B=b_0(a,d,r)$. The established threshold theorem
\eqref{eq:LD-threshold} says exactly that $\Delta(\beta)=0$ on $(0,B]$
and $\Delta(\beta)>0$ on $(B,\infty)$. Hence
\[
1-e^{-n\Delta(\beta)}\longrightarrow
\boldsymbol 1_{\{\beta>B\}}
\quad\text{as }n\to\infty.
\]
Applying dominated convergence once more to \eqref{eq:LD-inner-limit}
gives
\begin{equation}\label{eq:LD-outer-limit}
\lim_{n\to\infty}\lim_{m\to\infty}I_{n,m}
=\int_B^\infty e^{-\beta}\dd\beta=e^{-B}.
\end{equation}
Since $0<B<\infty$, taking the negative logarithm proves
\eqref{eq:LD-main}.
\end{proof}

\section{Comparison and scope of the two threshold representations}

The two thresholds concern disjoint regimes of the same differential
system. For $a>1$, one locates the unique zero of the bistable
wave speed by a signed excess-mass limit; for $0<a<1$, one locates the
boundary of the plateau where the minimal monostable wave speed equals
$2\sqrt{1-a}$, using a nonnegative speed gap. In particular,
$b_*(a,d,r)$ and $b_0(a,d,r)$ are not two formulas for the
same function, and neither construction makes an assertion at $a=1$.

\begin{enumerate}
\item The parameter formula \eqref{eq:LD-main} depends only on $a,d,r$.
The symbol $\beta$ is an integration variable. The functions $p_k,q_k$
in \eqref{eq:LD-recursion} are explicitly specified finite integral
expressions. Expanding the recursion leaves Gaussian kernels,
polynomials, the initial cosine function, and finite multiple integrals.

\item Allowing $b<1$ is essential to the full threshold question.
If $b_0<1$, the transition lies in weak competition and every
$b>1$ is nonlinearly selected. If $b_0=1$, the degenerate
case is the transition. If $b_0>1$, linear determinacy
persists on $1<b\le b_0$.

\item The two limits in \eqref{eq:LD-main} are ordered. The first recovers
the exact speed gap; the second tests whether that nonnegative gap is
strictly positive. The proof does not justify replacing them by an
arbitrary common sequence $n=n(m)\to\infty$. Such a replacement needs
additional control of the large-time spreading error, beyond the
explicit splitting estimate \eqref{eq:LD-diagonal}.

\item The expression is an exact constructive integral--limit
representation. It does not assert a formula containing a fixed finite
number of elementary or classical special functions. The threshold
theorem of Wu--Xiao--Zhou \cite{WXXZ} is an input; the Gaussian construction and
the parameter integration give the additional representation.
\end{enumerate}

\end{bluechange}

\end{document}